\documentclass[11pt,letterpaper,reqno]{amsart}
\usepackage{tikz}
\usetikzlibrary{positioning, shapes.geometric, arrows.meta}
\usepackage{amssymb}
\usepackage{amsmath}
\usepackage{amsthm}
\usepackage{amsfonts}
\usepackage{bbm}

\usepackage{pgfplots}
\pgfplotsset{compat=1.18} 

\usepackage{graphicx}
\usepackage[T1]{fontenc}
\usepackage{doi}
\usepackage{float} 
\usepackage{bookmark}
\usepackage{hyperref}
\hypersetup{pdfstartview={FitH}}
\newcommand{\C}{\mathbb{C}}
\newcommand{\cE}{\mathcal{E}}
\newcommand{\norm}[1]{\lVert #1 \rVert}
\newcommand{\abs}[1]{| #1 |}
\newcommand{\bv}{\mathbf{v}}
\newcommand{\bw}{\mathbf{w}}
\newcommand{\tr}{\text{tr}}

\newtheorem{thm}{Theorem}[section]
\newtheorem{lem}[thm]{Lemma}
\newtheorem{prop}[thm]{Proposition}
\newtheorem{cor}[thm]{Corollary}

\newtheorem{ques}[thm]{Question}
\newtheorem{conj}[thm]{Conjecture}
\theoremstyle{definition}
\newtheorem{exm}[thm]{Example}
\newtheorem{remark}[thm]{Remark}
\newtheorem{defn}[thm]{Definition}
\numberwithin{equation}{section}

\makeatother

\begin{document}

\title[A Spectral Lower Bound on Chromatic Numbers using $p$-Energy]{A Spectral Lower Bound on Chromatic Numbers using $p$-Energy}

\author[Clive Elphick]{Clive Elphick}
\author[Quanyu Tang]{Quanyu Tang}
\author[Shengtong Zhang]{Shengtong Zhang}

\address{School of Mathematics, University of Birmingham, Birmingham, UK}
\email{clive.elphick@gmail.com}

\address{School of Mathematics and Statistics, Xi'an Jiaotong University, Xi'an 710049, P. R. China}
\email{tang\_quanyu@163.com}

\address{Department of Mathematics, Stanford University, Stanford, CA 94305, USA}
\email{stzh1555@stanford.edu}



\begin{abstract}
Let $A_G $ be the adjacency matrix of a simple graph $ G $, and let $ \chi(G) $, $ \chi_f(G) $, $ \chi_q(G) $, $ \xi(G) $ and $ \xi_f(G) $ denote its chromatic number, fractional chromatic number, quantum chromatic number, orthogonal rank and projective rank, respectively. For $ p \geq 0 $, we define the positive and negative $ p $-energies of $ G $ by
$$
\mathcal{E}_p^+(G) = \sum_{\lambda_i > 0} \lambda_i^p, \quad 
\mathcal{E}_p^-(G) = \sum_{\lambda_i < 0} |\lambda_i|^p,
$$
where $ \lambda_1 \geq \cdots \geq \lambda_n $ are the eigenvalues of $A_G $. We prove that for all $ p \geq 0 $,
$$
\chi(G) \geq \left\{\chi_f(G), \chi_q(G), \xi(G) \right\} \geq \xi_f(G) \geq 1 + \max\left\{ \frac{\mathcal{E}_p^+(G)}{\mathcal{E}_p^-(G)}, \frac{\mathcal{E}_p^-(G)}{\mathcal{E}_p^+(G)} \right\}^{\frac{1}{|p - 1|}}.
$$
This result unifies and strengthens a series of existing bounds corresponding to the cases $ p \in \{0, 2, \infty\} $. In particular, the case $ p = 0 $ yields the inertia bound
$$
\chi_f(G)  \geq \xi_f(G) \geq1 + \max\left\{\frac{n^+}{n^-}, \frac{n^-}{n^+}\right\},
$$
where $ n^+ $ and $ n^- $ denote the number of positive and negative eigenvalues of $ A_G $, respectively. This resolves two conjectures of Elphick and Wocjan.

We also demonstrate that for certain graphs, non-integer values of $ p $ provide sharper lower bounds than existing spectral bounds. As an example, we determine $ \chi_q $ for the Tilley graph, which cannot be achieved using existing (unweighted) $p$-energy bounds. Our proof employs a novel synthesis of linear algebra and measure-theoretic tools, which allows us to surpass existing spectral bounds.
\end{abstract}

\maketitle

\section{Introduction}\label{secIntroduction}

We begin by recalling some fundamental concepts and notations. All graphs considered in this paper are assumed to be \emph{simple}, meaning they are undirected and contain no loops or multiple edges. Let \( G = (V, E) \) be a simple graph with \( n \) vertices and \( m \) edges. The \emph{adjacency matrix} of \( G \), denoted by \(A_G = (a_{ij})_{i,j=1}^n \), is an \( n \times n \) symmetric matrix where \( a_{ij} = 1 \) if the vertices \( v_i \) and \( v_j \) are adjacent, and \( a_{ij} = 0 \) otherwise. Since \(A_G \) is real symmetric, all its eigenvalues are real and can therefore be arranged in non-increasing order:
\[
\lambda_1(A_G) \geq \lambda_2(A_G) \geq \cdots \geq \lambda_n(A_G).
\]

For a real number \( p > 0 \), the \emph{\( p \)-energy} of \( G \) is defined as
\[
\mathcal{E}_p(G) = \sum_{i=1}^{n} |\lambda_i(A_G)|^p.
\]
When \( p = 1 \), this reduces to the classical graph energy \( \mathcal{E}(G) \), a quantity originally introduced in the context of theoretical chemistry. Over time, graph energy has become an active area of research in spectral graph theory; see \cite{Gutman2001} for a comprehensive survey. In recent years, the study of higher-order energies \( \mathcal{E}_p(G) \) has attracted increasing attention, leading to a variety of intriguing problems and conjectures; see~\cite{Elphick2016, Nikiforov2012, Nikiforov2016, Tang2025} for further details.

For a complex matrix \( X \in \mathbb{C}^{m \times n} \), let \( t = \min\{m, n\} \). We denote by \( s(X)=\{s_j(X)\}_{j=1}^t \) the sequence of \emph{singular values} of \( X \), i.e., the eigenvalues of the positive semi-definite matrix \( |X| = (X^*X)^{1/2} \), arranged in non-increasing order, where \( X^* \) stands for the \emph{conjugate transpose} of \( X \). For any \( p > 0 \), the \emph{Schatten \( p \)-norm} of \( X \) is defined by
\[
\|X\|_p = \left( \sum_{j=1}^t s_j(X)^p \right)^{1/p} = \left( \operatorname{tr} |X|^p \right)^{1/p},
\]
where \( \operatorname{tr} \) denotes the standard trace functional. This expression defines a norm on \( \mathbb{C}^{m \times n} \) for \( 1 \leq p < \infty \), and a quasi-norm when \( 0 < p < 1 \). If \( G \) is a graph with adjacency matrix \(A_G \), we denote \( \|G\|_p := \|A_G\|_p \) for brevity. Clearly, the \( p \)-energy \( \mathcal{E}_p(G) \) of a graph \( G \) satisfies \(\mathcal{E}_p(G) = \|G\|_p^p\). 


The \emph{inertia} of a graph \( G \) is the ordered triple \( (n^+, n^0, n^-) \), where \( n^+ \), \( n^0 \), and \( n^- \) denote the number (counted with multiplicities) of positive, zero, and negative eigenvalues of \(A_G \), respectively. For any real number \( p \geq 0 \), we define the \emph{positive \( p \)-energy} and \emph{negative \( p \)-energy} of \( G \) as
\[
\mathcal{E}_p^+(G) = \sum_{i=1}^{n^+} \lambda_i^p(A_G), \quad
\mathcal{E}_p^-(G) = \sum_{i=n - n^- + 1}^{n} |\lambda_i(A_G)|^p.
\] In particular, when \( p = 0 \), these quantities reduce to
\begin{equation}\label{eq1}
\mathcal{E}_0^+(G) = n^+, \quad \mathcal{E}_0^-(G) = n^-,    
\end{equation}that is, the positive and negative \( 0 \)-energies simply count the number of positive and negative eigenvalues of \(A_G \), respectively. The study of positive and negative \( p \)-energies of graphs was initiated by Tang, Liu, and Wang~\cite{Tang2024}, and was further developed more recently by Akbari, Kumar, Mohar, and Pragada~\cite{Akbari2025}.

The \emph{chromatic number} \( \chi(G) \) of a graph \( G \) is defined as the minimum number of colors needed to color its vertices such that no two adjacent vertices share the same color. Determining \( \chi(G) \), or even approximating it within reasonable accuracy, is known to be NP-hard~\cite{Garey1978, Hastad1999}. Consequently, much research has focused on establishing upper and lower bounds for \( \chi(G) \). While upper bounds are typically obtained by constructing explicit colorings, lower bounds require proving that no proper coloring with fewer colors exists.

We will also consider several well-studied variants and relaxations of the chromatic number \( \chi(G) \): the fractional chromatic number \( \chi_f(G) \), the quantum chromatic number \( \chi_q(G) \), the orthogonal rank \( \xi(G) \), and the projective rank \( \xi_f(G) \).

The fractional chromatic number of \( G \) is defined as follows.
\begin{defn}
    Let \( G \) be a graph on \( n \) vertices. Denote by \( \mathcal{I}(G) \) the set of all independent sets of \( G \), and for each vertex \( v \in V(G) \), let \( \mathcal{I}(G, v) \subseteq \mathcal{I}(G) \) denote the subset consisting of those independent sets that contain \( v \). A \emph{fractional coloring} is a map \( \sigma: \mathcal{I}(G) \to \mathbb{R}_{\geq 0} \) such that
    \[
    \sum_{J \in \mathcal{I}(G,v)} \sigma_J \geq 1, \quad \forall v \in V(G).
    \]
    The \emph{fractional chromatic number} \( \chi_f(G) \) is defined as the minimum value of
    \[
    \sum_{J \in \mathcal{I}(G)} \sigma_J
    \]
    over all fractional colorings \( \sigma \).
\end{defn}

The classical chromatic number \( \chi(G) \) corresponds to the case where \( \sigma_J \in \{0, 1\} \) for every \( J \in \mathcal{I}(G) \), and hence \( \chi(G) \geq \chi_f(G) \). It is known that fractional chromatic numbers can be significantly smaller than their classical counterparts; for example, Kneser graphs can have \( \chi_f(G) \leq 3 \) while \( \chi(G) \) can be arbitrarily large. 

The \emph{quantum chromatic number} \( \chi_q(G) \) was introduced by Cameron et al.~\cite{cameron07}. For a more detailed overview and discussion, we refer the reader to~\cite{Elphick2019}. In this paper, we present the following equivalent but purely combinatorial definition of the quantum chromatic number, due to~\cite[Definition~1]{mancinska162}. For a positive integer \( r \), let \( [r] \) denote the set \( \{1, 2, \ldots, r\} \). For \( d > 0 \), let \( I_d \) and \( O_d \) denote the identity and zero matrices in \( \mathbb{C}^{d \times d} \), respectively. When the dimension is clear from context, we simply write \( I \) and \( O \) to denote the identity and zero matrices, respectively.

\begin{defn}\label{def1}
A \emph{quantum $r$-coloring} of a graph $G=(V,E)$  is a collection of orthogonal projectors\footnote{Throughout this paper, we use the terms \emph{orthogonal projection matrix} and \emph{orthogonal projector} interchangeably.} $\{ P_{v,k} : v\in V, k\in [r]\}$ in 
$\C^{d\times d}$ such that
\begin{itemize}
\item for all vertices $v\in V$
\begin{eqnarray*}
\sum_{k\in[r]} P_{v,k} & = & I_d \quad\quad \mathrm{(completeness)} 
\end{eqnarray*}
\item for all edges $vw\in E$ and for all $k\in[r]$
\begin{eqnarray*}
P_{v,k} P_{w,k}           & = & O_d \quad\quad \mathrm{(orthogonality)}
\end{eqnarray*}
\end{itemize}
The \emph{quantum chromatic number} $\chi_q(G)$ is the smallest $r$ for which the graph $G$ admits a quantum $r$-coloring for some dimension $d>0$.
\end{defn}

The classical chromatic number \( \chi(G) \) corresponds to the case \( d = 1 \) in Definition~\ref{def1}. Therefore, it is straightforward that \( \chi(G) \ge \chi_q(G) \) for all graphs. For some graphs, \( \chi_q(G) \) can be exponentially smaller than \( \chi(G) \), and Ciardo~\cite{Ciardo2025} has discussed the likelihood that there exist graphs with $\chi_q(G) = 3$ and $\chi(G)$ unbounded. 


The definition of the orthogonal rank \( \xi(G) \) can be found in~\cite{Wocjan2019}. 
\begin{defn}
The \emph{orthogonal rank} of a graph \( G = (V, E) \) is the smallest positive integer \( \xi(G) \) such that there exists an \emph{orthogonal representation}, that is, a collection of non-zero column vectors \( x_v \in \mathbb{C}^{\xi(G)} \) for \( v \in V \) satisfying the orthogonality condition
\[
x_v^* x_w = 0 \quad \text{for all } vw \in E.
\]
\end{defn}

In this paper, we focus on the \emph{projective rank} \( \xi_f(G) \), which we define below.

\begin{defn}\label{def:projectiverank}
A \emph{\( d/r \)-representation} of a graph \( G = (V, E) \) is a collection of rank-\( r \), \( d \times d \) orthogonal projectors \( \{P_v\}_{v \in V(G)} \) such that \( P_v P_w = O \) for all edges \( vw \in E \). The \emph{projective rank} of \( G \) is defined as
\[
\xi_f(G) = \inf_{d, r} \left\{ \frac{d}{r} : G \text{ has a } d/r \text{-representation} \right\}.
\]
\end{defn}

The concept of projective rank was introduced by Man\v{c}inska and Roberson~\cite{Mancinska2016}. For additional background, properties, and applications, we refer the reader to~\cite{Hogben2017, Roberson2013}. Furthermore, we note that the orthogonal rank \( \xi(G) \) corresponds to the special case \( r = 1 \) in Definition~\ref{def:projectiverank}, that is,
\[
\xi(G) = \min \left\{ d : G \text{ admits a } d/1 \text{-representation} \right\}.
\]
It then follows that \( \xi_f(G) \leq \xi(G) \) for all graphs.

It is also worth noting that there exist graphs with exponential gaps between \( \xi_f(G) \) and \( \chi_f(G) \). For example, for the orthogonality graph \( \Omega(n) \), we have \( \xi_f(G) = n \) while \( \chi_f(G) \) grows exponentially in \( n \); see~\cite[Section~7]{Wocjan2019}.

The five graph parameters discussed above satisfy the following quantitative relationships: \begin{equation}\label{eq:fiverelationships}
\chi(G) \geq \max\left\{ \chi_q(G), \chi_f(G), \xi(G) \right\} \geq \min\left\{ \chi_q(G), \chi_f(G), \xi(G) \right\} \geq \xi_f(G).
\end{equation} Proofs of these inequalities can be found in, for example,~\cite[Section~6.1]{Mancinska2016} and~\cite[Section~2]{Wocjan2019}. Furthermore, Man\v{c}inska and Roberson showed that \( \xi(G) \) and \( \chi_q(G) \) are incomparable, as are \( \chi_f(G) \) and \( \chi_q(G) \), and \( \chi_f(G) \) and \( \xi(G) \); see~\cite{Mancinska2016, mancinska162} for further details.

For two graph parameters \( f \) and \( g \), we say that \( f \leq g \) if \( f(G) \leq g(G) \) for all graphs \( G \). The Hasse diagram in Figure~\ref{fig:hasse1}, adapted from~\cite[Fig.~1]{Mancinska2016}, illustrates the known partial order among the five graph parameters considered in this paper.

\begin{figure}[H] 
\begin{tikzpicture}
[vx/.style = {inner sep = 3pt, text height=1.5ex,text depth=.25ex}];

\def\hstep{-1};
\def\wstep{1.5};
\def\wstepn{1.1};

\foreach \i/\name in
         {1/$\chi$,2/$\chi_q$,3/$\xi_f$}{
\node[vx](\i) at (0,\i*\hstep) {\name};}

\node[vx](2l) at (-1*\wstep,2*\hstep) {$\chi_f$};
\node[vx](2r) at (\wstep,2*\hstep) {$\xi$};

\draw (1) -- (2) -- (3);
\draw (1) -- (2l) -- (3);
\draw (1) -- (2r) -- (3);

\end{tikzpicture}
\caption{Partial order of the five graph parameters.}
\label{fig:hasse1}
\end{figure}
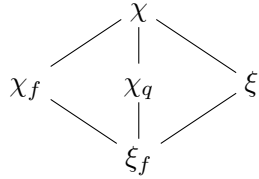

Spectral graph theory provides powerful tools for bounding the chromatic number, as the eigenvalues of the adjacency matrix encode global structural information about the graph, in contrast to local information such as vertex degrees.  One of the most celebrated spectral lower bounds is due to Hoffman~\cite{Hoffman1970}, which relates the chromatic number to the largest eigenvalue \( \lambda_1 \) and the smallest eigenvalue \( \lambda_n \) of the adjacency matrix:
\begin{equation}\label{eqhoffman}
\chi(G) \geq 1 + \frac{\lambda_1}{-\lambda_n}.
\end{equation}

In 2015, Ando and Lin~\cite{Ando2015} confirmed a conjecture of Wocjan and Elphick~\cite{Wocjan2013}, providing a novel spectral lower bound on the chromatic number using the positive and negative square energies of a graph:

\begin{thm}[\cite{Ando2015}]\label{prethm1}
Let \( \chi(G) \) be the chromatic number of a graph \( G \). Then
\[
  \chi(G) \geq 1 + \max\left\{ \frac{\mathcal{E}_2^+(G)}{\mathcal{E}_2^-(G)}, \frac{\mathcal{E}_2^-(G)}{\mathcal{E}_2^+(G)} \right\}.
\]
\end{thm}

Two years after the work of Ando and Lin, Elphick and Wocjan~\cite{Elphick2017} established the first spectral lower bound for the chromatic number that depends solely on the inertia of a graph. In light of \eqref{eq1}, this result can be reformulated in terms of the positive and negative \( 0 \)-energies as follows:

\begin{thm}[\cite{Elphick2017}]\label{prethm2}
Let \( \chi(G) \) be the chromatic number of a graph \( G \). Then
\[
  \chi(G) \geq 1 + \max\left\{ \frac{\mathcal{E}_0^+(G)}{\mathcal{E}_0^-(G)}, \frac{\mathcal{E}_0^-(G)}{\mathcal{E}_0^+(G)} \right\}
  = 1 + \max\left\{ \frac{n^+}{n^-}, \frac{n^-}{n^+} \right\}.
\]
\end{thm}

In~\cite{Elphick2019}, Elphick and Wocjan showed that many spectral lower bounds for \( \chi(G) \), including the Hoffman bound and the two bounds given in Theorems~\ref{prethm1} and~\ref{prethm2}, also apply to the quantum chromatic number \( \chi_q(G) \). Carli, Silva, Coutinho, and Grandsire~\cite{Silva2021}, as well as Guo and Spiro~\cite{Guo2024}, proved that both the Hoffman bound and Theorem~\ref{prethm1} also apply to the fractional chromatic number \( \chi_f(G) \). In~\cite{Wocjan2019}, Wocjan and Elphick showed that Theorem~\ref{prethm2} holds for the orthogonal rank \( \xi(G) \), and that the Hoffman bound also applies to both \( \xi(G) \) and \( \xi_f(G) \).

While the structures of these lower bounds look similar, previous works rely on ad-hoc methods to prove each individual bound. A natural and intriguing question arises:

\begin{ques}\label{ques1}
For \( p \geq 0 \), what can be said about the relationship between the quantities \( \chi(G) \), \( \chi_f(G) \), \( \chi_q(G) \), $\xi(G)$, $\xi_f(G)$, and
\[
  \max \left\{ \frac{\mathcal{E}_p^+(G)}{\mathcal{E}_p^-(G)}, \frac{\mathcal{E}_p^-(G)}{\mathcal{E}_p^+(G)} \right\}?
\]
\end{ques}

Note that the Hoffman bound can be viewed as the case \( p = \infty \) (see Proposition~\ref{prophoffman} for a detailed explanation). This question aims to unify the three aforementioned spectral lower bounds on the chromatic number and its variants, corresponding to the cases \( p = 0 \), \( p = 2 \), and \( p = \infty \), respectively. The main result of this paper provides a complete answer to this question for all \( p \geq 0 \).

\begin{thm}\label{thm3}
Let \( p \geq 0 \). Then for any non-empty simple graph $G$, we have
$$\chi(G) \geq \left\{\chi_f(G), \chi_q(G), \xi(G)\right\} \geq \xi_f(G) \geq 1 + \max\left\{ \frac{\mathcal{E}_p^+(G)}{\mathcal{E}_p^-(G)}, \frac{\mathcal{E}_p^-(G)}{\mathcal{E}_p^+(G)} \right\}^{\frac{1}{\abs{p - 1}}}.$$
\end{thm}

\begin{remark}
We note that the bound in Theorem~\ref{thm3} remains valid for all Hermitian weighted adjacency matrices of \( G \), without the need for any modification to the proof. Furthermore, this bound is tight for the full range of \( p \). For example, if \( G \) is the balanced Tur\'{a}n graph \( T_{rn, r} \), then \( \chi(G) = \chi_f(G) = \chi_q(G) = \xi(G) = \xi_f(G) = r \), \( \mathcal{E}_p^+(G) = (r - 1)^p n^p \), and \( \mathcal{E}_p^-(G) = (r - 1) n^p \). Hence, Theorem~\ref{thm3} provides an optimal answer to Question~\ref{ques1}.
\end{remark}


In~\cite[Conjecture~6]{Elphick2017} and~\cite[Section~6]{Wocjan2019}, Elphick and Wocjan conjectured that Theorem~\ref{prethm2} also holds for both \( \chi_f(G) \) and \( \xi_f(G) \). These conjectures remained open until the present work. As a consequence of our main result, we resolve both conjectures by taking \( p = 0 \) in Theorem~\ref{thm3}.
\begin{cor}
    For any non-empty simple graph $G$, we have
    $$\chi_f(G) \geq \xi_f(G) \geq 1 + \max\left\{\frac{n^+}{n^-}, \frac{n^-}{n^+}\right\}.$$
\end{cor}

We employ a unified proof technique that applies to the full range of \( p \geq 0 \). Our starting point is the matrix decomposition approach of Ando and Lin~\cite{Ando2015}, which expresses the adjacency matrix \( A_G \) as the difference of two positive semi-definite matrices whose entries agree at position \( (i, j) \) whenever \( ij \) is not an edge of \( G \). However, the remainder of their argument crucially depends on the identity that, for any symmetric matrix \( A \), the square energy \( \sum_{i = 1}^n \lambda_i(A)^2 \) equals the sum of the squares of the entries of \( A \). This identity no longer holds for non-integral values of \( p \). To overcome this limitation, we instead make use of tools from majorization theory to establish a family of key \( p \)-norm inequalities between certain matrices (see Proposition~\ref{lem:fractional-decomposition}). These inequalities allow us to generalize the Ando--Lin method to the full range \( p \geq 0 \).

The paper is organized as follows. In Section~\ref{sec2}, we introduce the necessary tools from linear algebra and measure theory, along with a family of key inequalities concerning the \( p \)-norm of matrices. Section~\ref{sec3} presents the proof of Theorem~\ref{thm3}. As an application of our result, Section~\ref{sec4} determines the exact value of \( \chi_q \) for the Tilley graph, for which previous \( p \)-energy bounds fail to apply. For various range of \( p \), we also present examples of graphs for which that specific value of \( p \) yields the optimal lower bound. In Section~\ref{secproperity}, we explore several properties of \( p \)-energies in the range \( 0 < p < 1 \). Finally, Section~\ref{secconclusion} presents some concluding remarks.

\section{Preliminaries}\label{sec2}
\subsection{Tools from linear algebra and measure theory}
To apply the technique of Ando--Lin, given a \( d/r \)-representation, we aim to construct a suitable analogue of a block decomposition. We now introduce a generalized notion of block decomposition.

\begin{defn}\label{def:continuousfractionalblockdecomposition}
For each \( n \in \mathbb{N} \), let \( \mathcal{P}_n \) denote the set of orthogonal projection matrices in \( \mathbb{C}^{n \times n} \). A \emph{stochastic block decomposition of size \( \alpha \)} is a matrix-valued random variable \( Q \colon \Omega \to \mathcal{P}_n \) for some \( n \), defined on a probability space \( (\Omega, \mathcal{F}, \mathbb{P}) \), such that
\[
\mathbb{E}\left[Q\right] = \alpha^{-1} I,
\]
where \( \mathbb{E}\left[Q\right] := \int_\Omega Q(\omega) \, d\mathbb{P}(\omega) \) denotes the entrywise expected value.
\end{defn}

To see how this definition relates to the classical notion of block decomposition, consider a decomposition of a square matrix \( M \) partitioned into \( r \times r \) blocks, \( M = (M_{ij})_{i,j=1}^r \). Let \( Q_i \) denote the orthogonal projector onto the indices corresponding to the $i$th diagonal block, and let \( Q \) be a random variable that takes values in \( \{ Q_1, \dots, Q_r \} \) with equal probability. Then
\[
\mathbb{E}\left[Q\right] = \frac{1}{r} (Q_1 + \cdots + Q_r) = \frac{1}{r} I,
\]
so \( Q \) is a stochastic block decomposition of size \( r \) in the sense of Definition~\ref{def:continuousfractionalblockdecomposition}. In an earlier version of this paper~\cite{Elphick2025}, we introduced the notion of \emph{fractional block decomposition}, which can also be seen as a special case of the stochastic block decomposition.

Before explaining how to generate a stochastic block decomposition from a \( d/r \)-representation, we first introduce the Haar measure on the unitary group \( \mathrm{U}(d) \). A detailed introduction to Haar measure tools in quantum information can be found in~\cite{Mele2024}.

\begin{defn}
The \emph{unitary group} \( \mathrm{U}(d) \) is defined as the set of all \( d \times d \) complex matrices \( U \in \mathbb{C}^{d \times d} \) such that
\[
U^* U = I_d.
\]
That is, \( \mathrm{U}(d) \) consists of all unitary matrices in \( \mathbb{C}^{d \times d} \).
\end{defn}

\begin{defn}The \emph{Haar measure on the unitary group} $\mathrm{U}(d)$ is the unique probability measure $\mu_H$ that is both left and right invariant over the group $\mathrm{U}(d)$, i.e., for all integrable functions $f$ and for all $V \in \mathrm{U}(d)$, we have:
\begin{align}
\int_{\mathrm{U}(d)} f\left(U\right)d\mu_H(U)=\int_{\mathrm{U}(d)} f\left(VU\right)d\mu_H(U)=\int_{\mathrm{U}(d)} f\left(UV\right)d\mu_H(U).
\end{align} Moreover, the Haar measure \( \mu_H \) is a probability measure, satisfying the properties
\[
\int_S 1 \, d\mu_H(U) \geq 0 \quad \text{for all measurable sets } S \subseteq \mathrm{U}(d), \quad \text{and} \quad \int_{\mathrm{U}(d)} 1 \, d\mu_H(U) = 1.
\] Hence, for any measurable function \( f \) on \( \mathrm{U}(d) \), we define its \emph{expected value} with respect to \( \mu_H \) as
\begin{align}
\underset{U\sim\mu_H}{\mathbb{E}}\left[ f(U) \right] := \int_{\mathrm{U}(d)} f(U) \, d\mu_H(U).
\end{align}
When \( f(U) \) is matrix-valued, the expectation is understood entrywise.
\end{defn}

Let \( \{e_v : v \in V(G)\} \) denote the standard basis of \( \mathbb{C}^n \), where \( n = |V(G)| \). Let the entries of the adjacency matrix \( A_G \) be denoted by \( a_{uv} \), where \( u, v \in V \) index the rows and columns, respectively.  
Then we have
\[
A_G = \sum_{v w \in E} a_{vw} \, e_v e_w^*,
\]
where \( a_{vw} = e_v^* A_G e_w \). We now describe how to generate a stochastic block decomposition from a \( d/r \)-representation.

\begin{lem}
\label{lem:quantum-to-block}
Let \( G = (V, E) \) be a graph on \( n \) vertices, and let \( \{P_v\}_{v \in V(G)} \) be a \( d/r \)-representation of \( G \). Let \( \Gamma \) be a random unitary matrix sampled from the Haar measure on \( \mathrm{U}(d) \), and define the random matrix
\[
Q = \sum_{v \in V(G)} e_v e_v^* \otimes \Gamma^* P_v \Gamma.
\]
Then \( Q \in \mathbb{C}^{n \times n} \otimes \mathbb{C}^{d \times d} \) is a stochastic block decomposition of size \( \frac{d}{r} \), and it satisfies \( Q (A_G \otimes I_d) Q = O \) almost surely (in fact, deterministically).
\end{lem}


\begin{proof}
Since each \( \Gamma^* P_v \Gamma \) is an orthogonal projector, it follows that \( Q^2 = Q = Q^* \), so \( Q \) is itself an orthogonal projector. Let \( \Phi(X) \) denote the expected value of \( \Gamma^* X \Gamma \) under the Haar measure on the unitary group \( \mathrm{U}(d) \):
\[
\Phi(X) = \underset{\Gamma\sim\mu_H}{\mathbb{E}}\left[\Gamma^* X \Gamma \right] = \int_{\mathrm{U}(d)} \Gamma^* X \Gamma \, d\mu_H(\Gamma),
\]
where \( d\mu_H(\Gamma) \) denotes the Haar measure on \( \mathrm{U}(d) \). We now compute the explicit form of $\Phi(X)$. The result, stated in Equation~\eqref{eq:firstmoment}, also appears in~\cite[Corollary~13]{Mele2024}. The map \( \Phi(X) \) is invariant under conjugation by any unitary matrix \( U \in \mathrm{U}(d) \), since
\[
U^* \Phi(X) U = \int_{\mathrm{U}(d)} (\Gamma U)^* X \Gamma U \, d\mu_H(\Gamma) = \int_{\mathrm{U}(d)} \Gamma^* X \Gamma \, d\mu_H(\Gamma) = \Phi(X),
\]
where the second equality follows from the right invariance of the Haar measure. Hence, \( \Phi(X) \) commutes with every unitary matrix \( U \in \mathrm{U}(d) \). By Schur's lemma (see~\cite[Theorem~1.2]{Berndt2007}), this implies that \( \Phi(X) \) must be a scalar multiple of the identity matrix:
\[
\Phi(X) = c I_d
\quad \text{for some scalar } c \in \mathbb{C}.
\]
Taking traces on both sides yields
\[
\mathrm{tr}(\Phi(X)) = \mathrm{tr}(X) = c \cdot \mathrm{tr}(I_d) = c d,
\]
so \( c = \frac{\mathrm{tr}(X)}{d} \), and thus
\begin{equation}\label{eq:firstmoment}
\Phi(X) = \frac{\mathrm{tr}(X)}{d} I_d.
\end{equation} Applying \eqref{eq:firstmoment} to \( X = P_v \), a rank-\( r \) orthogonal projector, we obtain
\[
\underset{\Gamma\sim\mu_H}{\mathbb{E}}\left[\Gamma^* P_v \Gamma \right] = \frac{r}{d} I_d.
\]
Therefore, we compute that
\[
\underset{\Gamma\sim\mu_H}{\mathbb{E}}\left[Q\right] = \sum_{v \in V} e_v e_v^* \otimes \underset{\Gamma\sim\mu_H}{\mathbb{E}}\left[\Gamma^* P_v \Gamma \right] = \frac{r}{d} \sum_{v \in V} e_v e_v^* \otimes I_d = \frac{r}{d} I_{nd}.
\] This confirms that \( Q \) is a stochastic block decomposition of size \( \frac{d}{r} \).

Finally, for each edge \( uv \in E(G) \), since \( P_u P_v = O \) in a \( d/r \)-representation, it follows that
\[
\Gamma^* P_v \Gamma \, \Gamma^* P_u \Gamma = \Gamma^* P_v P_u \Gamma = O,
\]
and hence
\[
Q (A_G \otimes I_d) Q = \sum_{uv \in E(G)} a_{uv} e_u e_v^* \otimes \Gamma^* P_v \Gamma \, \Gamma^* P_u \Gamma = O.
\]
This completes the proof.
\end{proof}

In addition, we introduce several notions and results from majorization theory. For a more detailed background, we refer the reader to~\cite[Part~I]{MOA11}.

\begin{defn}
We say that a vector \( \mathbf{x} \in \mathbb{R}^n \) is \emph{majorized} by a vector \( \mathbf{y} \in \mathbb{R}^n \) (i.e., \( \mathbf{y} \) \emph{majorizes} \( \mathbf{x} \)), denoted by \( \mathbf{x} \prec \mathbf{y} \), if the rearrangement of the components of \( \mathbf{x} \) and \( \mathbf{y} \) such that \(x_1 \geq x_2 \geq \cdots \geq x_n\), \(y_1 \geq y_2 \geq \cdots \geq y_n\) satisfies
\[
\sum_{i=1}^k x_i \leq \sum_{i=1}^k y_i \quad (1 \leq k \leq n - 1), \quad \text{and} \quad \sum_{i=1}^n x_i = \sum_{i=1}^n y_i.
\]
\end{defn}
\begin{remark}
For a set \( \Omega \subset \mathbb{R}^n \),  
we write \( x \prec y \) on \( \Omega \) to mean that \( x, y \in \Omega \) and \( x \prec y \).
\end{remark}

To derive inequalities from the majorization relation, the following concept will be useful.

\begin{defn}
A real-valued function \( f \) defined on a set \( \Omega \subseteq \mathbb{R}^n \) is said to be \emph{Schur-convex} on $\Omega$ if \[
x \prec y \quad \text{on } \Omega \quad \Rightarrow \quad \phi(x) \leq \phi(y).
\] Similarly, \( f \) is said to be \emph{Schur-concave} on $\Omega$ if \[
x \prec y \quad \text{on } \Omega \quad \Rightarrow \quad \phi(x) \geq \phi(y).
\]
\end{defn}
We now present two examples that will be used later. Their proofs can be found in~\cite[Section~3, C.1. Proposition]{MOA11}, and also follow from Karamata's inequality.
\begin{prop}
\label{prop:schur-concave}
Let \( F(x_1, x_2, \dots, x_n) = \sum_{i=1}^n x_i^p \) be defined on \( \mathbb{R}_+^n := \{ (x_1, \dots, x_n) \in \mathbb{R}^n : x_i \geq 0 \text{ for all } i \} \).
\begin{itemize}
    \item If \( 0 < p < 1 \), then \( F \) is Schur-concave on \( \mathbb{R}_+^n \).
    \item If \( p > 1 \), then \( F \) is Schur-convex on \( \mathbb{R}_+^n \).
\end{itemize}
\end{prop}

\subsection{$ p $-Norm inequalities for stochastic block decompositions}For a Hermitian positive semi-definite matrix \( X \in \mathbb{C}^{n \times n} \), let \( \lambda(X) = \{ \lambda_j(X) \}_{j=1}^n \) denote the sequence of eigenvalues of \( X \), arranged in non-increasing order. We set $\lambda_j(X) = 0$ for $j > n$. We are now ready to prove the following proposition, which replaces the entry-based analysis in Ando--Lin's method.
\begin{prop}
\label{lem:fractional-decomposition}
Let \( M \in \mathbb{C}^{n \times n} \) be a positive semi-definite matrix, and let \( Q \) be a stochastic block decomposition of size $\alpha$ in $\mathbb{C}^{n \times n} $. Define the sequence \( \beta(M) = (\beta_1, \dots, \beta_n) \) by
\[
\beta_j =\alpha \mathbb{E}\left[\lambda_j(QMQ) \right], \quad \text{for each } j = 1, \dots, n.
\]
Then:

\begin{itemize}
    \item[(1)] For any \( 0 < p < 1 \), we have
    \[
    \alpha \mathbb{E} \left[\|QMQ\|_p^p\right] \geq \|M\|_p^p \geq  \sum_{j=1}^n \beta_j^p \geq \alpha^p \mathbb{E} \left[\|QMQ\|_p^p\right].
    \]
    
    \item[(2)] For any \( p > 1 \), we have
    \[
    \alpha \mathbb{E} \left[\|QMQ\|_p^p\right] \leq \|M\|_p^p \leq  \sum_{j=1}^n \beta_j^p \leq \alpha^p \mathbb{E} \left[\|QMQ\|_p^p\right].
    \]
\end{itemize}
\end{prop}

To prove the middle and right-hand sides of both inequalities in Proposition~\ref{lem:fractional-decomposition}, we first present a lemma that establishes a majorization relation between \( \lambda(M) \) and \( \beta(M) \).

\begin{lem}\label{lem:RHS}
    Under the same assumptions of Proposition~\ref{lem:fractional-decomposition}, we have
    \[
    \lambda(M) \prec \beta(M),
    \]
    where \( \lambda(M) = (\lambda_1, \dots, \lambda_n) \) denotes the sequence of eigenvalues of \( M \) arranged in non-increasing order, and \( \beta(M) = (\beta_1, \dots, \beta_n) \) is defined as in Proposition~\ref{lem:fractional-decomposition}.
\end{lem}
\begin{proof}
    We first check that $\lambda(M)$ and $\beta(M)$ have the same sum. Note that
    $$\tr(Q M Q) = \tr(Q^2 M) = \tr(QM),$$
    hence
    $$\sum_{j=1}^n \beta_j = \alpha \sum_{j = 1}^n \mathbb{E} \left[ \lambda_j(QMQ)\right] = \alpha  \mathbb{E} \left[ \sum_{j = 1}^n \lambda_j(QM)\right] = \alpha  \mathbb{E} \left[ \tr(QM)\right]  = \alpha \tr \left( \mathbb{E} \left[Q\right]M\right) =\tr(M),$$ where the last equality uses the identity \( \mathbb{E}\left[Q\right] = \alpha^{-1} I \) from Definition~\ref{def:continuousfractionalblockdecomposition}.
    
    We next check that for any $ 1 \leq k \leq n-1$, we have
    $$\lambda_1 + \cdots + \lambda_k \leq \beta_1 + \cdots + \beta_k.$$
    To prove this, let \( M = \sum_{i = 1}^n \lambda_i \bv_i \bv_i^* \) be the spectral decomposition of \( M \), where the eigenvalues satisfy \( \lambda_1 \geq \lambda_2 \geq \cdots \geq \lambda_n \). Define \( N = \sum_{i = 1}^k \lambda_i \bv_i \bv_i^* \). Then we have
    $$\lambda_1 + \cdots + \lambda_k = \tr(N)= \alpha \tr \left( \mathbb{E}\left[Q\right]N\right) = \alpha  \mathbb{E} \left[ \tr(QN)\right] = \alpha  \mathbb{E} \left[ \tr(QNQ)\right].$$ As $N$ has rank at most $k$, it follows that $QNQ$ has rank at most $k$, and hence $\tr(QNQ) = \sum_{j = 1}^k \lambda_j(QNQ).$ Since \( M \) is positive semi-definite, we have \( M \geq N \geq O \). Thus,  we know that \( Q M Q - Q N Q = Q (M - N) Q \geq O.\) It follows that \( QMQ \geq QNQ \). By \cite[Corollary~4.3.12]{Horn13}, we know that\[
\lambda_j(QMQ) \geq \lambda_j(QNQ) \quad \text{for all } j.
\] Now we conclude that $$\lambda_1 + \cdots + \lambda_k = \alpha  \mathbb{E} \left[ \sum_{j = 1}^k \lambda_j(QNQ)\right] \leq \alpha  \mathbb{E} \left[ \sum_{j = 1}^k \lambda_j(QMQ)\right] = \sum_{j = 1}^k \alpha  \mathbb{E} \left[ \lambda_j(QMQ)\right]=\sum_{j = 1}^k \beta_j .$$
    Therefore, \( \lambda(M) \prec \beta(M) \).
\end{proof}

To prove the left-hand side of both inequalities in Proposition~\ref{lem:fractional-decomposition}, we establish the following lemma.
\begin{lem}\label{lem:LHS}
Let \( M \in \mathbb{C}^{n \times n}\) be a positive semi-definite matrix, and let \( Q \in \mathbb{C}^{n \times n} \) be an orthogonal projector. Then:
\begin{itemize}
    \item[(1)] For any real number \( 0 < p < 1 \), we have
    \[
    \operatorname{tr}\left((QMQ)^p \right) \geq \operatorname{tr}\left(QM^pQ\right).
    \]

     \item[(2)] For any real number \( p > 1 \), we have
    \[
    \operatorname{tr}\left((QMQ)^p\right) \leq \operatorname{tr}\left(QM^pQ\right).
    \]
\end{itemize}
\end{lem}
\begin{proof}
Let \( (\bv_i)_{i = 1}^n \) be an orthonormal eigenbasis of \( M \) with corresponding eigenvalues \( \lambda_i \). Let \( (\bw_j)_{j = 1}^n \) be an orthonormal eigenbasis of \( QMQ \), with eigenvalues \( \mu_j \), such that the vectors \( \bw_1, \dots, \bw_k \) span the image of the projector \( Q \). Define
\[
U_{ij} = \langle \bv_i, \bw_j \rangle = \bv_i^* \bw_j
\]
to be the entries of the change-of-basis matrix from \( \left(\bv_i\right)_{i = 1}^n \) to \( \left(\bw_j\right)_{j = 1}^n \). Then for each \( j \in [k] \), we have
\[
\mu_j = \bw_j^* M \bw_j  = \bw_j^*\left( \sum_{i=1}^{n} \lambda_i \bv_i \bv_i^*  \right)\bw_j= \sum_{i = 1}^n\lambda_i \abs{U_{ij}}^2.
\]
Therefore,
\[
\operatorname{tr}\left((QMQ)^p\right) = \sum_{j = 1}^k \mu_j^p = \sum_{j = 1}^k \left( \sum_{i = 1}^n \lambda_i \abs{U_{ij}}^2 \right)^p.
\]
On the other hand,
\[
\operatorname{tr}\left(Q M^p Q\right) = \sum_{j = 1}^k \bw_j^* M^p \bw_j = \sum_{j = 1}^k \bw_j^* \left( \sum_{i=1}^{n} \lambda_i^p \bv_i \bv_i^*  \right) \bw_j
= \sum_{j = 1}^k \sum_{i = 1}^n \lambda_i^p \abs{U_{ij}}^2.
\] Now, by Parseval's identity, for each \( j \), we have
    $$\sum_{i = 1}^n\abs{U_{ij}}^2 = \sum_{i = 1}^n |\langle \bv_i, \bw_j \rangle|^2 = \norm{\bw_j}^2 =1.$$ For $0 < p < 1$, applying Jensen's inequality with the concave function \(t^p\), we obtain
\[
\left( \sum_{i = 1}^n \lambda_i \abs{U_{ij}}^2 \right)^p \geq \sum_{i = 1}^n \lambda_i^p \abs{U_{ij}}^2.
\]
Summing over \( j = 1, \dots, k \), this yields
\[
\operatorname{tr}((QMQ)^p) \geq \operatorname{tr}(Q M^p Q).
\] Similarly, for \( p > 1 \), the function \( t^p \) is convex, so Jensen's inequality gives
\[
\left( \sum_{i = 1}^n \lambda_i \abs{U_{ij}}^2 \right)^p \leq \sum_{i = 1}^n \lambda_i^p \abs{U_{ij}}^2,
\]
and thus
\[
\operatorname{tr}((QMQ)^p) \leq \operatorname{tr}(Q M^p Q).
\]This completes the proof.\end{proof}
\begin{remark}
In fact, for \( 0 < p < 1 \) and \( 1 < p \leq 2 \), one can establish a stronger result than Lemma~\ref{lem:LHS}. When \( 1 < p \leq 2 \), it is known that the function \( t^p \) is operator convex; see~\cite[Exercise~V.2.11]{Bhatia13}. Therefore, by~\cite[Theorem~V.2.3]{Bhatia13}, we have
\[
(QMQ)^p \leq QM^p Q.
\] On the other hand, for \( 0 < p < 1 \), the function \( t^p \) is operator concave; see~\cite[Theorem~V.1.9]{Bhatia13} and~\cite[Theorem~V.2.5]{Bhatia13}. Hence, by~\cite[Theorem~V.2.3]{Bhatia13}, we obtain
\[
(QMQ)^p \geq QM^p Q.
\] However, for \( p > 2 \), a similar inequality does not hold in general, as \( t^p \) is no longer operator convex in this range; see~\cite[Exercise~V.2.11]{Bhatia13}.
\end{remark}

We are now ready to complete the proof of Proposition~\ref{lem:fractional-decomposition}.
\begin{proof}[Proof of Proposition~\ref{lem:fractional-decomposition}]
Assume first that \( 0 < p < 1 \). The middle side inequalities follow from Lemma~\ref{lem:RHS} together with Proposition~\ref{prop:schur-concave}. It thus remains to prove the right-hand and left-hand side inequalities.

For the right-hand side, we apply the probabilistic form of Jensen's inequality and obtain
\[
\left( \mathbb{E} \left[X\right] \right)^p \geq \mathbb{E} \left[X^p \right],
\]which holds for every integrable, real-valued random variable \( X \) defined on a probability space. Thus, we know that
$$\sum_{j=1}^n \left( \mathbb{E} \left[ \lambda_j(QMQ) \right] \right)^p \geq \sum_{j=1}^n  \mathbb{E} \left[ \lambda_j^p(QMQ) \right],
$$ which means $\sum_{j=1}^n \beta_j^p \geq \alpha^p \mathbb{E} \left[\|QMQ\|_p^p\right]$.

For the left-hand side, by Lemma~\ref{lem:LHS}, we have
$$\norm{Q M Q}_p^p = \tr \left((Q M Q)^p\right) \geq \tr(Q M^p Q).$$ Taking the expected value and multiplying by $\alpha$ on both sides, we obtain
\[
\alpha \mathbb{E} \left[ \|QMQ\|_p^p \right] \geq \alpha \mathbb{E} \left[ \tr(Q M^p Q) \right] = \alpha \mathbb{E} \left[ \tr(Q M^p) \right] = \alpha \tr \left( \mathbb{E} \left[ Q\right] M\right)
= \operatorname{tr}(M^p) = \|M\|_p^p.
\] where the second-to-last equality uses the identity \( \mathbb{E}\left[Q\right] = \alpha^{-1} I \) from Definition~\ref{def:continuousfractionalblockdecomposition}.

The case \( p > 1 \) follows from a similar argument, using the reversed inequality in Lemma~\ref{lem:LHS}. This completes the proof.
\end{proof}

\section{Proof of the main result}\label{sec3}
In view of Inequality~\eqref{eq:fiverelationships}, to prove Theorem~\ref{thm3}, it suffices to establish the following lower bound for \( \xi_f(G) \).

\begin{thm}\label{thmxifmcase}
    Let \( p \geq 0 \). Then for any non-empty simple graph $G$, we have
    $$\xi_f(G) \geq 1 + \max\left\{\frac{\cE_p^+(G)}{\cE_p^-(G)}, \frac{\cE_p^-(G)}{\cE_p^+(G)}\right\}^{\frac{1}{\abs{p - 1}}}.$$
\end{thm}
\begin{proof}
Let \( A_G \in \mathbb{R}^{n \times n} \) denote the adjacency matrix of the graph \( G \), and let its eigenvalues be ordered as \[
\lambda_1 \geq \lambda_2 \geq \cdots \geq \lambda_n.
\] Let \( n^+ \) and \( n^- \) denote the numbers of positive and negative eigenvalues of \( A_G \), respectively. For each \( i = 1, \dots, n \), let \( v_i \) be an eigenvector corresponding to \( \lambda_i \), and assume that \( \{v_1, \dots, v_n\} \) forms an orthonormal basis of \( \mathbb{R}^n \). Then, by the spectral decomposition of \( A_G \), we have
\[
A_G = \sum_{i=1}^{n} \lambda_i v_i v_i^*.
\]
Define
\begin{equation}\label{eqpnpenergy}
A_G^+ = \sum_{i = 1}^{n^+} \lambda_i v_i v_i^*, \quad 
A_G^- = -\sum_{i = n-n^{-}+1}^{n} \lambda_i v_i v_i^*.
\end{equation}
Then both \( A_G^+ \) and \( A_G^- \) are positive semi-definite matrices, and they satisfy \( A_G^+ A_G^- = A_G^- A_G^+ = O \). 

Let \( \{P_v\}_{v \in V(G)} \) be a \( d/r \)-representation of \( G \), and sample \( \Gamma \) from the Haar measure on \( \mathrm{U}(d) \). Then, by Lemma~\ref{lem:quantum-to-block}, the random matrix
\[
Q = \sum_{v \in V} e_v e_v^* \otimes \Gamma^* P_v \Gamma
\]
is a stochastic block decomposition of size \( \frac{d}{r} \), and satisfies
\[
Q \left(A_G \otimes I_d \right) Q = O.
\] Let \( \epsilon > 0 \) be a parameter to be chosen later, and define
\[
B = (1 + \epsilon) A_G^+, \quad C = \epsilon A_G^+ + A_G^-.
\]
We also define the lifted matrices
\[
\widetilde{B} = B \otimes I_d = \left((1 + \epsilon) A_G^+ \right) \otimes I_d, \quad 
\widetilde{C} = C \otimes I_d = \left(\epsilon A_G^+ + A_G^- \right) \otimes I_d.
\]
Then we have the identity
\begin{equation}\label{eqeq:partitionidentity}
Q \widetilde{B} Q = Q \widetilde{C} Q.
\end{equation}since \( \widetilde{B} - \widetilde{C} = A_G \otimes I_d \) and \( Q \left(A_G \otimes I_d \right) Q = O\).

\subsection{Case \texorpdfstring{$0 < p < 1$}{0 < p < 1}}
By part~(1) of Proposition~\ref{lem:fractional-decomposition} and the identity~\eqref{eqeq:partitionidentity}, we obtain \begin{equation}\label{0<p<1quantumpenergy11}\|\widetilde{B}\|_p^p \geq \left( \frac{d}{r} \right)^{p} \mathbb{E} \left[\|Q\widetilde{B}Q\|_p^p\right] = \left( \frac{d}{r} \right)^{p} \mathbb{E} \left[\|Q\widetilde{C}Q\|_p^p\right] \geq \left( \frac{d}{r} \right)^{p-1} \|\widetilde{C}\|_p^p.\end{equation}
Moreover, since
\[
\|\widetilde{B}\|_p^p = \| B\|_p^p  \|I_d \|_p^p = d \cdot \| B \|_p^p,
\quad \text{and} \quad
\|\widetilde{C}\|_p^p = \| C\|_p^p  \|I_d \|_p^p = d \cdot \| C \|_p^p,
\]
we may cancel the common factor \( d \) in Equation~\eqref{0<p<1quantumpenergy11}, yielding
\begin{equation} \label{eq:xi_f_bound_step1}
\left( \frac{d}{r} \right)^{1 - p} \|B\|_p^p \geq  \|C\|_p^p.
\end{equation}
Taking the infimum over all possible pairs \( (d, r) \) in~\eqref{eq:xi_f_bound_step1}, we obtain
\begin{equation}\label{eqxi_f1}
\xi_f(G)^{1 - p} \|B\|_p^p \geq \|C\|_p^p.
\end{equation}We compute that\begin{equation}\label{eqxi_f2}
    \|B\|_p^p = (1 + \epsilon)^p \cE_p^+(G), \quad \|C\|_p^p = \epsilon^p \cE_p^+(G) + \cE_p^-(G).\end{equation}
    Thus, by \eqref{eqxi_f1} and \eqref{eqxi_f2}, we have
    $$\xi_f(G)^{1 - p} (1 + \epsilon)^p \cE_p^+(G) \geq \epsilon^p \cE_p^+(G) + \cE_p^-(G).$$
    Taking $\epsilon = \frac{1}{\xi_f(G) - 1}$, we obtain
    $$\xi_f(G)^{1 - p} \frac{\xi_f(G)^p}{(\xi_f(G) - 1)^p} \cE_p^+(G) \geq \frac{1}{(\xi_f(G) - 1)^p} \cE_p^+(G) + \cE_p^-(G).$$
    Hence, we know that
    $$(\xi_f(G) - 1)^{1 - p} \cE_p^+(G) \geq \cE_p^-(G),$$ which means $$\xi_f(G) \geq 1+\left(\frac{\cE_p^-(G)}{\cE_p^+(G)}\right)^{\frac{1}{1-p}}.$$ Similarly, reversing the roles of the positive and negative parts in the preceding argument yields\[
\xi_f(G) \geq 1 + \left( \frac{\mathcal{E}_p^+(G)}{\mathcal{E}_p^-(G)} \right)^{\frac{1}{1 - p}}.
\]Therefore, \begin{equation}\label{eeqqmajor6}\xi_f(G) \geq 1 + \max\left\{\frac{\cE_p^+(G)}{\cE_p^-(G)}, \frac{\cE_p^-(G)}{\cE_p^+(G)}\right\}^{\frac{1}{1-p}}\end{equation} holds for all $0<p<1$.

\subsection{Case \texorpdfstring{$ p > 1$}{ p > 1}}
By part~(2) of Proposition~\ref{lem:fractional-decomposition} and the identity~\eqref{eqeq:partitionidentity}, we obtain \begin{equation}\label{p>1quantumpenergy11}\|\widetilde{B}\|_p^p \leq \left( \frac{d}{r} \right)^{p} \mathbb{E} \left[\|Q\widetilde{B}Q\|_p^p\right] = \left( \frac{d}{r} \right)^{p} \mathbb{E} \left[\|Q\widetilde{C}Q\|_p^p\right] \leq \left( \frac{d}{r} \right)^{p-1} \|\widetilde{C}\|_p^p.\end{equation}
Moreover, since \( \|\widetilde{B}\|_p^p = d \|B\|_p^p \) and \( \|\widetilde{C}\|_p^p = d \|C\|_p^p \), we may cancel the common factor \( d \) in Inequality~\eqref{p>1quantumpenergy11}. Taking the infimum over all possible pairs \( (d, r) \), we obtain
\begin{equation}\label{eqxi_f3}
\|B\|_p^p \leq \xi_f(G)^{p-1} \|C\|_p^p.
\end{equation}Combining~\eqref{eqxi_f2} and~\eqref{eqxi_f3}, we obtain
\[
\xi_f(G)^{1 - p} (1 + \epsilon)^p \, \mathcal{E}_p^+(G) \leq \epsilon^p \, \mathcal{E}_p^+(G) + \mathcal{E}_p^-(G).
\]Setting \( \epsilon = \frac{1}{\xi_f(G) - 1} \) yields $$\xi_f(G)^{1 - p} \frac{\xi_f(G)^p}{(\xi_f(G) - 1)^p} \cE_p^+(G) \leq \frac{1}{(\xi_f(G) - 1)^p} \cE_p^+(G) + \cE_p^-(G).$$
    Hence, we know that
    $$ \cE_p^+(G) \leq (\xi_f(G) - 1)^{p-1}\cE_p^-(G),$$ which means $$\xi_f(G) \geq 1+\left(\frac{\cE_p^+(G)}{\cE_p^-(G)}\right)^{\frac{1}{p-1}}.$$ Similarly, reversing the roles of the positive and negative parts in the preceding argument yields $$\xi_f(G) \geq 1+\left(\frac{\cE_p^-(G)}{\cE_p^+(G)}\right)^{\frac{1}{p-1}}.$$
Therefore, $$\xi_f(G) \geq 1 + \max\left\{\frac{\cE_p^+(G)}{\cE_p^-(G)}, \frac{\cE_p^-(G)}{\cE_p^+(G)}\right\}^{\frac{1}{p-1}}$$ holds for all $p>1$.

\subsection{Case $ p \in \{0, 1\} $}

Taking the limit as \( p \to 0^+ \) in inequality~\eqref{eeqqmajor6} yields the desired conclusion for \( p = 0 \). It therefore remains to justify the case \( p \to 1 \). To this end, define
\[
H(p) := \max\left\{ \frac{\mathcal{E}_p^+(G)}{\mathcal{E}_p^-(G)}, \frac{\mathcal{E}_p^-(G)}{\mathcal{E}_p^+(G)} \right\}^{\frac{1}{|p - 1|}}.
\]
We aim to show that the limit \( \lim_{p \to 1} H(p) \) exists. Set \(R(p) := \mathcal{E}_p^+(G)/\mathcal{E}_p^-(G)\), then \(H(p) = \max\left\{ R(p), R(p)^{-1} \right\}^{\frac{1}{|p - 1|}}\). 

Now we treat \( \mathcal{E}_p^+(G) \) as a function of \( p \). Its first-order Taylor expansion around \( p = 1 \) is given by
\[
\mathcal{E}_p^+(G) = \mathcal{E}_1^+(G) + (p - 1) \left. \frac{d}{dp} \mathcal{E}_p^+(G) \right|_{p=1} + o(p - 1)
= \mathcal{E}_1^+(G) \left[ 1 + (p - 1) \alpha_+ + o(p - 1) \right],
\]
where \[
\alpha_+ :=\frac{1}{\mathcal{E}_1^+(G)}  \left. \frac{d}{dp} \mathcal{E}_p^+(G) \right|_{p=1} =\frac{1}{\mathcal{E}_1^+(G)} \sum_{\lambda_i > 0} \lambda_i \log \lambda_i.
\]
Similarly, we have the expansion
\[
\mathcal{E}_p^-(G) = \mathcal{E}_1^-(G) \left[ 1 + (p - 1) \alpha_- + o(p - 1) \right],
\]
where 
\[
\alpha_- := \frac{1}{\mathcal{E}_1^-(G)} \sum_{\lambda_i < 0} |\lambda_i| \log |\lambda_i|.
\]Since \( \mathcal{E}_1^+(G) = \mathcal{E}_1^-(G) \), we obtain
\[
R(p) = \frac{\mathcal{E}_1^+(G) \left[ 1 + (p - 1)\alpha_+ + o(p - 1) \right]}
       {\mathcal{E}_1^-(G) \left[ 1 + (p - 1)\alpha_- + o(p - 1) \right]} 
= \frac{1 + (p - 1)\alpha_+ + o(p - 1)}
       {1 + (p - 1)\alpha_- + o(p - 1)}.
\]
Now applying the first-order expansion of the function \( \frac{1 + \varepsilon a}{1 + \varepsilon b} \approx 1 + \varepsilon(a - b) \), we obtain
\[
R(p) = 1 + (p - 1)(\alpha_+ - \alpha_-) + o(p - 1).
\]
Hence,
\[
\log R(p) = \log \left(1 + (p - 1)(\alpha_+ - \alpha_-) + o(p - 1)\right)
= (p - 1)(\alpha_+ - \alpha_-) + o(p - 1).
\]
Finally, we compute
\[
\log H(p) = \frac{1}{|p - 1|} \log \max \{ R(p), R(p)^{-1} \}
= \frac{1}{|p - 1|}  \abs{\log R(p)}
= \abs{ \alpha_+ - \alpha_- } + o(1),
\]
which implies
\[
\lim_{p \to 1} H(p) = \exp{|\alpha_+ - \alpha_-|}.
\]
Therefore, the limit exists. This completes the proof of Theorem~\ref{thmxifmcase}.\end{proof}

 \section{Comparisons with existing spectral bounds}\label{sec4}

\subsection{Derivation of $\chi_q$ for the Tilley graph}
The quantum chromatic number \( \chi_q(G) \) is always an integer, but it is not known to be computable in general. As a result, spectral bounds can play a valuable role in bounding or even determining \( \chi_q(G) \). For instance, in \cite{Elphick2017} it is shown that \( \chi_q(\mathrm{Clebsch}) = 4 \) and that \( \chi_q(\mathrm{GQ}(2,4)) \in \{5,6\} \). Further developments by \cite[Section~3]{Wocjan2020} established that $\chi_q(\mathrm{GQ}(2,4)) = 6$, $\chi_q(\mathrm{Kneser}_{p,2}) = p - 2$, $\chi_q(\mathrm{Gewirtz}) = 4$, and $\chi_q(\mathrm{Higman\text{-}Sims}) \in \{5,6\}$.

Our method is based on the observation that \( \chi_q(G) = \chi(G) \) if the spectral lower bound given in Theorem~\ref{thm3}, evaluated at an optimal value of \( p \), exceeds \( \chi(G) - 1 \). We restrict our analysis to graphs with \( \chi(G) > \omega(G) \), where \( \omega(G) \) denotes the clique number, since it is known that \(\omega(G) \le \chi_q(G) \le \chi(G).\)

\begin{exm}
The irregular \emph{Tilley graph} $H_1$, which arises in the study of Kempe chains, has \( \left|V(H_1)\right| = 12 \), \( \omega(H_1) = 3 \), and \( \chi(H_1) = 4 \), with spectrum approximately given by
\begin{align*}
\{4.86272,\ 1+\sqrt{2},\ \sqrt{5},\ 1.90542,\ 1-\sqrt{2},\ -1,\ -1,\ -1,\ -1.32557,\ -2,\ -\sqrt{5},\ -2.44258\}.
\end{align*}We compute
\[
   1+\max_{m = 1, \dots, n - 1}\left\{
    \frac{\mathcal{E}_2^{\pm}(H_1)}{\mathcal{E}_2^{\mp}(H_1)},\,
    \frac{n^{\pm}(H_1)}{n^{\mp}(H_1)},\, \frac{\sum_{i = 1}^{m} \lambda_i(H_1)}{- \sum_{i = 1}^{m} \lambda_{n - i + 1}(H_1)}
    \right\}
    = 1+\frac{\lambda_1(H_1)}{-\lambda_n(H_1)} \approx 2.99082.
\]This shows that the Hoffman bound and its extensions given in Wocjan and Elphick~\cite[Theorem~1]{Wocjan2013}, Ando and Lin~\cite{Ando2015}, Elphick and Wocjan~\cite{Elphick2017} all yield only \( \chi_q(H_1) \geq 3 \). However, for \( p = 13.3466 \), we have \[
    f(p;H_1) = 1 + \max\left\{ \frac{\mathcal{E}_p^+(H_1)}{\mathcal{E}_p^-(H_1)}, \frac{\mathcal{E}_p^-(H_1)}{\mathcal{E}_p^+(H_1)} \right\}^{\frac{1}{|p - 1|}} \approx 3.05114 > 3 >2.99082.
    \]Thus, Theorem~\ref{thm3} implies \( \chi_q(H_1) \geq 4 \). Since $\chi(H_1)=4$, we conclude that $\chi_q(H_1)=4$. Figure~\ref{Tilley-graph2} shows the graph of \(f(p;H_1)\) as a function of \( p \).

\begin{figure}[h]
\centering
\begin{tikzpicture}
  \begin{axis}[
    width=10cm,
    height=6cm,
    xlabel={$p$},
    ylabel={$f(p;H_1)$},
    domain=0.01:25,
    samples=200,
    thick,
    axis lines=middle,
    grid=both,
    xmin=0, xmax=25,
    ymin=2.905, ymax=3.06,
    enlargelimits=true,
    legend pos=north east]
    \addplot[blue, thick]
{
  1+(
    max(
      (4.86272^x + 2.41421^x + 2.23607^x + 1.90542^x)
      /
      (abs(-0.414214)^x + abs(-1)^x + abs(-1)^x + abs(-1)^x + abs(-1.32557)^x + abs(-2)^x + abs(-2.23607)^x + abs(-2.44258)^x),
      (abs(-0.414214)^x + abs(-1)^x + abs(-1)^x + abs(-1)^x + abs(-1.32557)^x + abs(-2)^x + abs(-2.23607)^x + abs(-2.44258)^x)
      /
     (4.86272^x + 2.41421^x + 2.23607^x + 1.90542^x)
    )
  )^(1 / abs(x - 1))
};    
  \end{axis}
\end{tikzpicture}
\caption{A plot of the function $f(p;H_1)$ over \( p \in [0, 25] \).}
\label{Tilley-graph2}
\end{figure}
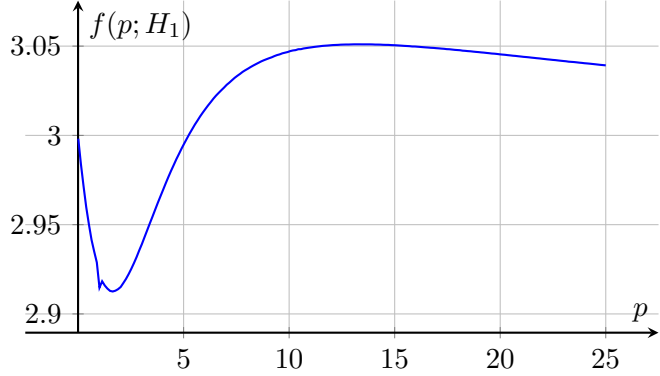 

\end{exm}


\subsection{Full range of optimal \texorpdfstring{$p$}{p}}
\begin{exm}Consider the 7-vertex graph \( H_2 \) shown in Figure~\ref{fig:example-graph1}, whose adjacency matrix has spectrum approximately given by
\begin{align*}
\{3.38896,\ 1.33155,\ 0,\ -0.638678,\ -1,\ -1,\ -2.08183\}.
\end{align*}

\begin{figure}[!h]
\centering
\begin{tikzpicture}[main_node/.style={circle,draw,minimum size=0.1em,inner sep=3pt]}]

\node[main_node] (0) at (-4.817752982124082, 4.83408983216184) {};
\node[main_node] (1) at (-4.857142857142858, 2.766525996190758) {};
\node[main_node] (2) at (-2.685268270621277, 4.857142857142858) {};
\node[main_node] (3) at (-2.727046739975563, 2.6618992901064287) {};
\node[main_node] (4) at (-0.46134669239627213, 3.7176265182215853) {};
\node[main_node] (5) at (1.8273856830402726, 4.722388032085533) {};
\node[main_node] (6) at (1.751919675111557, 2.5565477539376458) {};

 \path[draw, thick]
(0) edge node {} (2) 
(2) edge node {} (3) 
(1) edge node {} (3) 
(1) edge node {} (0) 
(0) edge node {} (3) 
(1) edge node {} (2) 
(4) edge node {} (2) 
(4) edge node {} (3) 
(5) edge node {} (4) 
(4) edge node {} (6) 
;

\end{tikzpicture}
\caption{The 7-vertex graph \( H_2 \)}
\label{fig:example-graph1}
\end{figure}
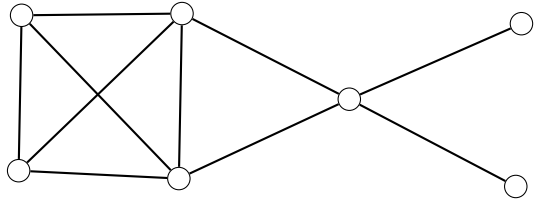
\noindent We compute
\[
   1+\max_{m = 1, \dots, n - 1}\left\{
   \frac{\mathcal{E}_2^{\pm}(H_2)}{\mathcal{E}_2^{\mp}(H_2)},\,
   \frac{n^{\pm}(H_2)}{n^{\mp}(H_2)},\, \frac{\sum_{i = 1}^{m} \lambda_i(H_2)}{- \sum_{i = 1}^{m} \lambda_{n - i + 1}(H_2)}
   \right\} = 1+\frac{n^-(H_2)}{n^+(H_2)} = 3.
   \]  However, for \( p = 0.562125 \), we have 
   \[
    f(p;H_2) =1+\max\left\{\frac{\mathcal{E}_{p}^+(H_2)}{\mathcal{E}_{p}^-(H_2)}, \frac{\mathcal{E}_{p}^-(H_2)}{\mathcal{E}_{p}^+(H_2)}\right\}^{\frac{1}{|p - 1|}} \approx 3.0064 > 3.
   \] Thus, for the graph \( H_2 \), a non-integer value of \( p \) performs better than any of the classical cases \( p \in \{0, 2, \infty\} \). Figure~\ref{fig:example-graph2} shows the graph of \(f(p;H_2)\) as a function of \( p \).

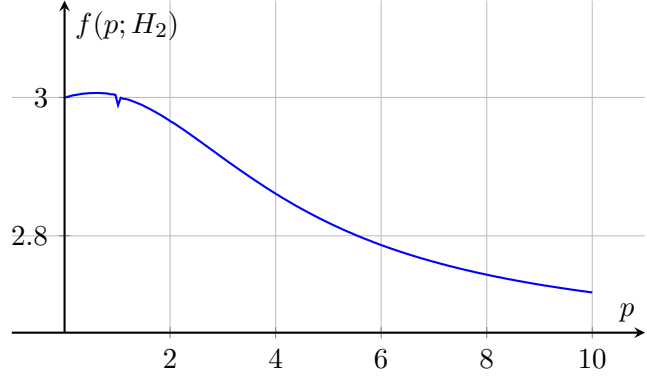
\begin{figure}[h]
\centering
\begin{tikzpicture}
  \begin{axis}[
    width=10cm,
    height=6cm,
    xlabel={$p$},
    ylabel={$f(p;H_2)$},
    domain=0.01:10,
    samples=200,
    thick,
    axis lines=middle,
    grid=both,
    xmin=0, xmax=10,
    ymin=2.7, ymax=3.1,
    enlargelimits=true,
    legend pos=north east]
    \addplot[blue, thick]
{
  1+(
    max(
      (3.38896^x + 1.33155^x)
      /
      (abs(-0.638678)^x + abs(-1)^x + abs(-1)^x + abs(-2.08183)^x),
      (abs(-0.638678)^x + abs(-1)^x + abs(-1)^x + abs(-2.08183)^x)
      /
      (3.38896^x + 1.33155^x)
    )
  )^(1 / abs(x - 1))
};    
  \end{axis}
\end{tikzpicture}
\caption{A plot of the function $ f(p;H_2)$ over \( p \in [0, 10] \).}
\label{fig:example-graph2}
\end{figure} 

\end{exm}

Now we let
\[
f(p;G) = 1 + \max\left\{ \frac{\mathcal{E}_p^+(G)}{\mathcal{E}_p^-(G)}, \frac{\mathcal{E}_p^-(G)}{\mathcal{E}_p^+(G)} \right\}^{\frac{1}{|p - 1|}}.
\]
Then Theorem~\ref{thm3} implies that
\[
\chi(G)  \geq \xi_f(G) \geq f(p;G) \quad \text{for all } p \geq 0.
\]Therefore, to obtain the optimal spectral lower bound, it suffices to compute \( \sup_{p \geq 0} f(p;G) \). Table \ref{tab:optimalp} presents example graphs, each of which achieves the maximum of \( f(p;G) \) within a different range of \( p \).

\begin{table}[H]
\centering
\begin{tabular}{|l|c|l|c|c|c|}
\hline
\textbf{Reference} & \textbf{Optimal \( p \)} & \textbf{Example graph \( G \)} & \textbf{ \( f(p;G) \)} & \textbf{$\left\lceil f(p;G) \right\rceil$} & \textbf{\( \chi(G) \)} \\
\hline
Elphick--Wocjan~\cite{Elphick2017} & \( p = 0 \) & Generalised Quadrangle \( \mathrm{GQ}(2,4) \) & 4.5 & 5 & 5 \\
This paper & \( 0 < p < 1 \) & Graph $H_2$ in Figure~\ref{fig:example-graph1} & 3.0064 & 4 & 4 \\
This paper & \( 1 < p < 2 \) & Beineke’s non-line graph \( G_5 \) & 3.1296 & 4 & 4 \\
Ando--Lin~\cite{Ando2015} & \( p = 2 \) & Regular complete \( r \)-partite graphs & \( r \) & \( r \) & \( r \) \\
This paper & \( 2 < p < \infty \) & Tilley graph & 3.0515 & 4 & 4 \\
Hoffman~\cite{Hoffman1970} & \( p \to \infty \) & Generalised Quadrangle \( \mathrm{GQ}(2,1) \) & 3 & 3 & 3 \\
\hline
\end{tabular}
\caption{
Graphs that are optimal for different values of \( p \) in the spectral lower bound. The first column indicates the earliest paper where the optimal \( p \) bound was established. The fourth column gives the maximum value of \( f(p;G) \), with its ceiling and the actual chromatic number shown in the last two columns.
}
\label{tab:optimalp}
\end{table}

\section{Properties of \texorpdfstring{$p$}{p}-energies when \texorpdfstring{$0 < p < 1$}{0 < p < 1} }\label{secproperity}
Several properties of the \( p \)-energy, as well as the positive and negative \( p \)-energies, have been discussed for \( p \geq 1 \) in~\cite[Section~4.1]{Arizmendi2023} and~\cite{Akbari2025}, respectively. In this section, we investigate analogous properties of the \( p \)-energy, as well as the positive and negative \( p \)-energies, in the range \( 0 < p < 1 \).

The following monotonicity property of the \( \ell_p \)-norm is well known:

\begin{lem}\label{propMonotonicity}
Let \( 0 < p < q < \infty \), and let \( a_1, a_2, \dots, a_n \) be positive real numbers. Then
\begin{equation}\label{eqschatten1}
\left( \sum_{i=1}^n a_i^p \right)^{1/p} \geq \left( \sum_{i=1}^n a_i^q \right)^{1/q}.    
\end{equation}
\end{lem}

Moreover, the following classical bound on the 1-energy of a graph can be found in~\cite[Theorem~5.2]{Li2012}.

\begin{lem}[\cite{Li2012}]\label{lem1energy}
Let \( G \) be a graph with \( m \) edges. Then
\[
2\sqrt{m} \leq \mathcal{E}_1(G) \leq 2m.
\]
\end{lem}

Analogous to~\cite[Corollary~4.3]{Arizmendi2023}, we now present a lower bound for the \( p \)-energy in the range \( 0 < p < 1 \), expressed in terms of the number of edges.

\begin{prop}\label{proppenergyedgebound}
Let \( 0 < p < 1 \), and let \( G \) be a graph with \( m \) edges. Then
\[
\mathcal{E}_p(G) \geq 2m^{p/2}.
\]
\end{prop}
\begin{proof}
By Lemma~\ref{lem1energy}, we have
\[
\mathcal{E}_1^+(G) = \mathcal{E}_1^-(G) = \frac{1}{2} \mathcal{E}_1(G) \geq \sqrt{m}.
\]
Then, applying Lemma~\ref{propMonotonicity}, it follows that
\begin{equation}\label{eqedgeppp}
\mathcal{E}_p^+(G) \geq \left( \mathcal{E}_1^+(G) \right)^p \geq \left( \sqrt{m} \right)^p = m^{p/2}.
\end{equation}
Similarly, we obtain the same lower bound for \( \mathcal{E}_p^-(G) \). Therefore,
\[
\mathcal{E}_p(G) = \mathcal{E}_p^+(G) + \mathcal{E}_p^-(G) \geq  2m^{p/2}.
\] This completes the proof.\end{proof}

Inspired by a conjecture proposed by Elphick, Wocjan, Farber, and Goldberg in~\cite[Conjecture~1]{Elphick2016}, we propose the following \( p \)-energy analogue:

\begin{conj}\label{ewfgpenergy}
Let \( 0 \leq p \leq 2 \), and let \( G \) be a connected graph with \( n \) vertices. Then
\begin{equation}\label{ewfgeq}
\min\left\{\mathcal{E}_p^+(G) , \mathcal{E}_p^-(G)\right\} \geq (n - 1)^{p/2}.
\end{equation}
\end{conj}
\begin{remark}
When \( p = 2 \), Conjecture~\ref{ewfgpenergy} coincides with the original conjecture proposed in~\cite[Conjecture~1]{Elphick2016}. For the case \( p = 1 \), the inequality follows immediately from Lemma~\ref{lem1energy}. The case \( p = 0 \) is trivial.
\end{remark}
\begin{remark}
When \(0 < p < 2\), Lemma~\ref{propMonotonicity} implies that
\[
(\mathcal{E}_p^+(G))^{1/p} \geq (\mathcal{E}_2^+(G))^{1/2}, \quad \text{and} \quad (\mathcal{E}_p^-(G))^{1/p} \geq (\mathcal{E}_2^-(G))^{1/2}.
\]
Therefore, Conjecture~\ref{ewfgpenergy} is in fact equivalent to \cite[Conjecture~1]{Elphick2016}. Nevertheless, proving the conjecture for any specific value of \(p < 2\) is still valuable; for example, the case \(p = 1.5\) might be more tractable.
\end{remark}

We are now in a position to address the first half of Conjecture~\ref{ewfgpenergy}, namely, to present a sharp lower bound for the positive and negative \( p \)-energies when \( 0 < p < 1 \).

\begin{thm}\label{analogelphick}
Let \( 0 < p < 1 \), and let \( G \) be a connected graph with \( n \) vertices and \( m \) edges. Then
\begin{equation}\label{sharpp1}
\min\left\{\mathcal{E}_p^+(G) , \mathcal{E}_p^-(G)\right\} \geq (n - 1)^{p/2}.
\end{equation}
\end{thm}
\begin{proof}
By~\eqref{eqedgeppp}, we obtain
\[
\mathcal{E}_p^+(G) \geq m^{p/2}.
\]
Since \( G \) is connected, it follows that \( m \geq n - 1 \), and hence
\[
\mathcal{E}_p^+(G) \geq (n - 1)^{p/2}.
\]
A similar argument yields the same lower bound for \( \mathcal{E}_p^-(G) \). This completes the proof.
\end{proof}
\begin{remark}
Inequality~\eqref{sharpp1} is sharp. For instance, equality is attained when \( G \) is the star graph \( S_n \).    
\end{remark}

\section{Concluding remarks}\label{secconclusion}

In this paper, we established a unified family of lower bounds for the chromatic number \( \chi(G) \), the fractional chromatic number \( \chi_f(G) \), the quantum chromatic number \( \chi_q(G) \), the orthogonal rank \(\xi(G) \) and the projective rank \(\xi_f(G) \), all expressed in terms of the positive and negative \( p \)-energies of a graph \( G \). In certain cases, these bounds strictly improve upon existing spectral bounds. As a direct corollary, we resolve two conjectures posed by Elphick and Wocjan in~\cite[Conjecture~6]{Elphick2017} and~\cite[Section~6]{Wocjan2019}.

In Section~\ref{secIntroduction}, we mentioned that Theorem~\ref{thm3}, in the limit as \( p \to \infty \), recovers the classical Hoffman bound. We now formally prove this result:

\begin{prop}\label{prophoffman}
Theorem~\ref{thm3} implies the Hoffman bound given in \eqref{eqhoffman}.
\end{prop}
\begin{proof}
For \( p>1 \), Theorem~\ref{thm3} implies the bound
\[
\xi_f(G) \geq 1 + \left( \frac{\mathcal{E}_p^+(G)}{\mathcal{E}_p^-(G)} \right)^{\frac{1}{p - 1}} = 1 + \left( \frac{\mathcal{E}_p^+(G)}{n^+}\frac{n^+}{n^-}\frac{n^-}{\mathcal{E}_p^-(G)} \right)^{\frac{1}{p}\frac{p}{p - 1}}.
\] It is well known that the \( p \)th power mean of a sequence of positive numbers converges to its maximum as \( p \to \infty \). Thus, we have
\[
\left( \frac{\mathcal{E}_p^+(G)}{n^+} \right)^{1/p} \to \lambda_1, \quad
\left( \frac{\mathcal{E}_p^-(G)}{n^-} \right)^{1/p} \to -\lambda_n \quad \text{as } p \to \infty,
\]
where \( \lambda_1 \) and \( \lambda_n \) are the largest and smallest eigenvalues of the adjacency matrix of \( G \), respectively.
Therefore, we have
\[
1 + \left( \frac{\mathcal{E}_p^+(G)}{n^+}  \frac{n^+}{n^-}  \frac{n^-}{\mathcal{E}_p^-(G)} \right)^{\frac{1}{p} \frac{p}{p - 1}}
\to 1 + \frac{\lambda_1}{-\lambda_n}
\quad \text{as } p \to \infty.
\]
This completes the proof. \end{proof}


Therefore, this work can also be viewed as a unification and strengthening of three classical spectral lower bounds on the chromatic number, namely those of Hoffman~\cite{Hoffman1970}, Ando and Lin~\cite{Ando2015}, and Elphick and Wocjan~\cite{Elphick2017}.


It is worth noting that our bound in Theorem~\ref{thm3} cannot be extended to the case \( p < 0 \). As a counterexample, consider the 5-cycle \( C_5 \). For \( p = -4 \), a direct computation yields
\[
1 + \left( \frac{\mathcal{E}_{-4}^+(C_5)}{\mathcal{E}_{-4}^-(C_5)} \right)^{1/5} > 3.1,
\]
which contradicts the fact that \( \chi(C_5) = 3 \).

Our proof does not follow the elegant method of Guo and Spiro~\cite{Guo2024}, as their approach appears to rely crucially on the assumption that \( p = 2 \). Nonetheless, one may attempt to generalize \cite[Theorem~3]{Guo2024} to more values of \( p \).

\begin{ques}
    For which values of \( p \geq 0 \) does the following hold:  
    if there exists a homomorphism from a graph \( G \) to an edge-transitive graph \( H \), then
\begin{equation}\label{eq:conclusionguo1}
        \frac{\lambda_{\max}(H)}{\abs{\lambda_{\min}(H)}} \geq \max\left\{
        \frac{\mathcal{E}_p^+(G)}{\mathcal{E}_p^-(G)},
        \frac{\mathcal{E}_p^-(G)}{\mathcal{E}_p^+(G)}
        \right\}^{\frac{1}{\abs{p - 1}}}.
\end{equation}
\end{ques} We remark that inequality~\eqref{eq:conclusionguo1} cannot hold for all values of \( p \geq 0 \). For instance, if \( G = H = C_5 \), then the inequality fails for the entire interval \( p \in [0, 1.154] \).

The vector chromatic number satisfies \( \chi_v(G) \le \xi_f(G) \)~\cite{Wocjan2019}. This naturally raises the question of whether \( \chi_v(G) \) can replace \( \xi_f(G) \) in Theorem~\ref{thm3}. The answer is negative: for example, the cycle \( C_5 \) serves as a counterexample for all \(0 \leq  p \leq 1.154 \). On the other hand, Coutinho, Spier, and Zhang~\cite{Coutinho2024} proved that Theorem~\ref{thm3} does provide a valid lower bound on \( \chi_v(G) \) in the case \( p = 2 \). Furthermore, Galtman~\cite{Galtman2000} and independently Bilu~\cite{Bilu2006} showed that the Hoffman bound is also a lower bound for the vector chromatic number, and thus Theorem~\ref{thm3} provides a valid lower bound on \( \chi_v(G) \) when \( p = \infty \). Motivated by these results, we propose the following conjecture:

\begin{conj}\label{conj:vector-chromatic-bound}
For all \( p \geq 2 \), the vector chromatic number of any graph \( G \) satisfies
\[
\chi_v(G) \geq 1 + \max\left\{ \frac{\mathcal{E}_p^+(G)}{\mathcal{E}_p^-(G)}, \frac{\mathcal{E}_p^-(G)}{\mathcal{E}_p^+(G)} \right\}^{\frac{1}{p - 1}}.
\]
\end{conj}

We have tested Conjecture~\ref{conj:vector-chromatic-bound} on all graphs with at most 9 vertices, as well as on all named graphs with at most 100 vertices in the Wolfram Mathematica database, using code available in the GitHub repository~\cite{repo}. No counterexamples have been found.

\section*{Acknowledgments}
 The second author would like to express his sincere gratitude to Prof.~Minghua Lin and Yinchen Liu for numerous helpful discussions and valuable suggestions. He is also grateful to Prof.~Leonid Chekhov and Prof.~Michael Shapiro at Michigan State University for their support.  

This material is based upon work supported by the National Science
Foundation under Grant No. DMS-1928930, while the third author was in
residence at the Simons Laufer Mathematical Sciences Institute in
Berkeley, California, during the Spring 2025 semester.

 The authors would like to thank Prof.~Aida Abiad for her comments on an earlier version of this paper. The authors would also like to thank Prof.~Terence Tao for helpful comments on Lemma~\ref{lem:RHS}.

\end{document}